\RequirePackage[l2tabu, orthodox]{nag}
\documentclass[11pt]{article}
\usepackage{amsmath,amsthm,amsfonts,amssymb,enumerate,calc,graphicx,float,wrapfig,caption}
\usepackage[UKenglish]{babel}
\usepackage[colorlinks=true,citecolor=black,linkcolor=black,urlcolor=blue]{hyperref}
\usepackage[noabbrev,capitalise]{cleveref}
\usepackage[lmargin=30mm,rmargin=30mm,bmargin=30mm,tmargin=30mm]{geometry}
\usepackage[numbers,sort&compress]{natbib}
\renewcommand{\baselinestretch}{1.1}
\usepackage[hang]{footmisc} 

\renewcommand{\thefootnote}{\fnsymbol{footnote}}	
\allowdisplaybreaks
\newcommand\DateFootnote{
\begingroup
\renewcommand\thefootnote{}
\footnote{4th August 2016, Revised: \today}
\setcounter{footnote}{0}
\vspace*{-3ex}
\endgroup}

\makeatletter
\renewcommand\section{\@startsection {section}{1}{\z@}%
                                   {-3ex \@plus -1ex \@minus -.2ex}%
                                   {2ex \@plus.2ex}%
                                   {\normalfont\large\bfseries}}
\renewcommand\subsection{\@startsection{subsection}{2}{\z@}%
                                     {-2.5ex\@plus -1ex \@minus -.2ex}%
                                     {1.5ex \@plus .2ex}%
                                     {\normalfont\normalsize\bfseries}}
\renewcommand\subsubsection{\@startsection{subsubsection}{3}{\z@}%
                                     {-2ex\@plus -1ex \@minus -.2ex}%
                                     {1ex \@plus .2ex}%
                                     {\normalfont\normalsize\bfseries}}
 \renewcommand\paragraph{\@startsection{paragraph}{4}{\z@}%
                                    {1.5ex \@plus.5ex \@minus.2ex}%
                                    {-1em}%
                                    {\normalfont\normalsize\bfseries}}
\renewcommand\subparagraph{\@startsection{subparagraph}{5}{\parindent}%
                                       {1.5ex \@plus.5ex \@minus .2ex}%
                                       {-1em}%
                                      {\normalfont\normalsize\bfseries}}
\makeatother

\newcommand{\msn}[1]{MR:\,\href{http://www.ams.org/mathscinet-getitem?mr=MR#1}{#1}}
\newcommand{\MSN}[2]{MR:\,\href{http://www.ams.org/mathscinet-getitem?mr=MR#1}{#1}}
\newcommand{\doi}[1]{doi:\,\href{http://dx.doi.org/#1}{#1}}

\newcommand{\Zbl}[1]{Zbl:\,\href{http://www.zentralblatt-math.org/zmath/en/search/?q=an:#1}{#1}}

\theoremstyle{plain}
\newtheorem{thm}{Theorem}
\newtheorem{lem}[thm]{Lemma}

\newtheorem{cor}[thm]{Corollary}
\newtheorem{prop}[thm]{Proposition}
\newtheorem{claim}[thm]{Claim}
\theoremstyle{definition}

\newcommand{\ceil}[1]{\lceil{#1}\rceil}
\newcommand{\floor}[1]{\lfloor{#1}\rfloor}
\newcommand{\CEIL}[1]{\left\lceil{#1}\right\rceil}

\renewcommand{\geq}{\geqslant}
\renewcommand{\leq}{\leqslant}

\newcommand{\authcite}[1]{\mbox{\citet{#1}}}

\begin{document}

{\Large\bfseries\boldmath\scshape Edge-Maximal Graphs on Surfaces}

\DateFootnote

{\large 
Colin McDiarmid\,\footnotemark[2]
\quad 
David~R.~Wood\,\footnotemark[3]
}

\footnotetext[2]{Department of Statistics, University of Oxford, United Kingdom (\texttt{cmcd@stats.ox.ac.uk}).}

\footnotetext[3]{School of Mathematical Sciences, Monash University, Melbourne, Australia (\texttt{david.wood@monash.edu}). \\
Research supported by  the Australian Research Council.}

\emph{Abstract.} We prove that for every surface $\Sigma$ of Euler genus $g$, every edge-maximal embedding of a graph in $\Sigma$ is at most $O(g)$ edges short of a triangulation of $\Sigma$. This provides the first answer to an open problem of Kainen (1974). 

\section{Introduction}
\label{Intro}
\renewcommand{\thefootnote}{\arabic{footnote}}

For a graph class $\mathcal{G}$, a graph $G\in\mathcal{G}$ is \emph{edge-maximal} if adding any non-edge to $G$ produces a graph not in $\mathcal{G}$. We emphasise that ``graph'' here means a simple graph with no parallel edges and no loops. A graph class $\mathcal{G}$ is \emph{pure} if $|E(G)|=|E(H)|$ for all edge-maximal graphs $G,H\in\mathcal{G}$  with $|V(G)|=|V(H)|$. For example, each of the following graph classes is pure: 
\begin{wrapfigure}{r}{68mm}
\vspace*{2.5ex}
\centering
\includegraphics{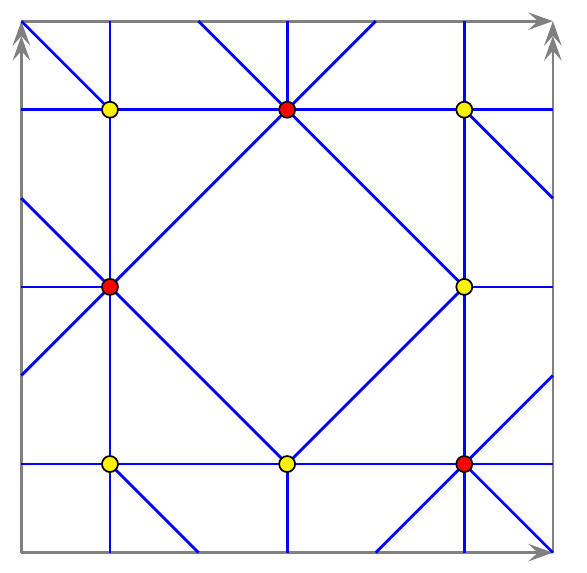}
\captionsetup{width=65mm}
\caption{\label{K8C5} An embedding of $K_8-E(C_5)$ in the torus. Every such embedding has one 4-face, which induces $K_4$, so no non-edge can be added.}
\vspace*{-1ex}
\end{wrapfigure}
forests, outerplanar graphs, planar graphs; and for each positive integer $k$, the $k$-degenerate  graphs, the graphs of treewidth at most $k$, and the chordal graphs with clique number at most $k+1$ (where the last two classes have the same edge-maximal members, the $k$-trees). On the other hand, toroidal graphs are not pure: \authcite{HKSW73} proved that $K_8-E(C_5)$ is an edge-maximal toroidal graph but is not a toroidal triangulation (see \cref{K8C5}). 

Motivated by this example, \authcite{Kainen74} posed the following open problem: by how many edges can an edge-maximal graph embeddable in a given surface fail to be a triangulation? This paper addresses this natural question, which surprisingly has been ignored in the literature. We prove that for every surface $\Sigma$ of Euler genus $g$, every edge-maximal graph embeddable in $\Sigma$ is $O(g)$ edges short of a triangulation (regardless of the number of vertices).

We formulate this result as follows. A graph class $\mathcal{G}$ is \emph{$k$-impure} if $||E(G)|-|E(H)||\leq k$ for all edge-maximal graphs $G,H\in\mathcal{G}$ with $|V(G)|=|V(H)|$. For $h\geq 0$, let $\mathbb{S}_h$ be the sphere with $h$ handles. For  $c\geq 0$, let $\mathbb{N}_c$ be the sphere with $c$ cross-caps. Every surface is homeomorphic to $\mathbb{S}_h$ or $\mathbb{N}_c$. The \emph{Euler genus} of $\mathbb{S}_h$ is $2h$. The \emph{Euler genus} of $\mathbb{N}_c$ is $c$. The \emph{Euler genus} of a graph $G$ is the minimum Euler genus of a surface in which $G$ embeds. See \citep{MoharThom} for definitions and background about graphs embedded in surfaces. The following is our main theorem; see Theorems~\ref{NonOrient} and \ref{orientable} for fuller forms of this result.

\begin{thm} 
\label{Summary}
The class of graphs embeddable in a surface $\Sigma$ of Euler genus $g$ is $O(g)$-impure. 
\end{thm}

To add some perspective to this result, note that several interesting graph classes are not at all pure. Consider, for example, the $K_5$-minor-free graphs. The 8-vertex Mobius ladder is $K_5$-minor-free with 12 edges. Pasting copies of this graph on edges produces a $K_5$-minor-free graph with $n\equiv 2\pmod{6}$ vertices and $(11n-16)/6$ edges. It is edge-maximal with no $K_5$-minor by Wagner's characterisation \citep{Wagner37}. On the other hand, every $n$-vertex edge-maximal planar graph is edge-maximal with no $K_5$-minor, yet has $3(n-2)$ edges. Thus the difference between the number of edges in these two classes of edge-maximal $K_5$-minor-free graphs grows with $n$, and indeed is $\Omega(n)$. In general, $K_t$-minor-free graphs can have as many as $ct\sqrt{\log t}\,n$ edges \citep{Thomason01,Thomason84,Kostochka84}, but there are edge-maximal $K_t$-minor-free graphs, namely $(t-2)$-trees, with only $(t-2)n-\binom{t-1}{2}$ edges (for $n \geq t-1$).

Let $\mathcal{G}_H$ denote the class of graphs not containing $H$ as a minor. \authcite{MP16} proved that (ignoring $K_1$) the only connected graphs $H$ such that $\mathcal{G}_H$ is pure are $K_2$, $K_3$, $K_4$ and $P_3$ (the 3-vertex path). Furthermore, for each connected graph $H$, either $\mathcal{G}_H$ is $k$-impure for some $k$, or there are $n$-vertex graphs $G_n$ and $G'_n$ in $\mathcal{G}_H$ such that $|E(G_n)|-|E(G'_n)|$ is $\Omega(n)$.

\section{Main Proof}

An embedding of a graph $G$ in a surface is \emph{edge-maximal} if for every non-edge $e$ of $G$, it is not possible to add $e$ to the embedding (without changing the embedding of $G$). Observe that an embedding of a graph $G$ in a surface is edge-maximal if and only if for each face $F$, the set of vertices on $F$ induce a clique in $G$. Also note that a graph $G$ is edge-maximal embeddable in a surface $\Sigma$ if and only if \emph{every} embedding of $G$ in $\Sigma$ is edge-maximal. We mentioned above that Theorems~\ref{NonOrient} and \ref{orientable}  give fuller forms of \cref{Summary}; in fact, they concern edge-maximal embeddings (as well as giving explicit constants). The distinction between edge-maximal embeddings and edge-maximal graphs is exemplified by the following fact. An embedding is \emph{2-cell} (or \emph{cellular}) if each face is homeomorphic to an open disc. 
 
\begin{prop}
\label{planar}
For each surface $\Sigma$, there are infinitely many planar graphs, each with an edge-maximal 2-cell embedding in $\Sigma$.
\end{prop}

\emph{Proof.} 
First suppose that $\Sigma=\mathbb{N}_g$.  Let $G_0$ be a triangulation of the sphere with at least $g$ faces. 
Say  $F_1,\dots,F_g$ are distinct faces  of $G_0$.    
Note that $K_4$ has a 2-cell embedding in the projective plane with two triangular 
faces and one face of length 6 (see \cref{K4}).  Let $Q_1,\dots,Q_g$ 
 \begin{wrapfigure}{r}{76mm}
\centering
\includegraphics{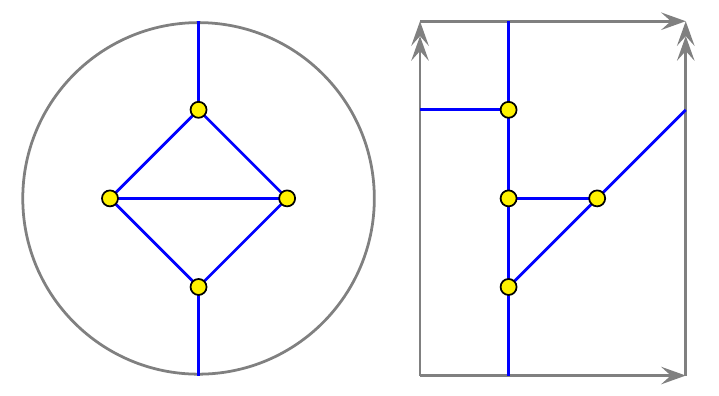}
\captionsetup{width=73mm}
\vspace*{1ex}
\caption{\label{K4} Embeddings of $K_4$ in $\mathbb{N}_1$ and $\mathbb{S}_1$.}
\end{wrapfigure}
be $g$ copies of this embedding of $K_4$. For $i\in[1,g]$, identify $F_i$ with a triangular face of $Q_i$.  We obtain a graph $G$ embedded in $\mathbb{N}_g$, in which each face induces a clique. Thus this embedding of $G$ is edge-maximal. Note that $G$ is a planar triangulation, since it is obtained from $G_0$ by simply adding a degree-3 vertex inside $g$ faces of $G_0$. An analogous proof works for $\Sigma=\mathbb{S}_h$ since $K_4$ has a 2-cell embedding in the torus with one triangular face and one face of length 9 (see \cref{K4}). 
\qed

A \emph{pseudograph} is a graph possibly with parallel edges and loops. 
A (\emph{pseudograph}) \emph{triangulation} is a 2-cell embedded (pseudo)graph in which each face has length exactly 3. Euler's formula implies that every pseudograph with $n\geq 3$ vertices that embeds in a surface of Euler genus $g$ such that each face has length at least 3 has at most $3(n+g-2)$ edges, with equality if and only if the embedding is a pseudograph triangulation. Of course, every face in an embedding of a graph has length at least 3. Thus every graph with $n\geq 3$ vertices that embeds in a surface of Euler genus $g$ has at most $3(n+g-2)$ edges, with equality if and only if the embedding is a triangulation.  Also note that Euler's formula implies that every bipartite graph with $n\geq 3$ vertices that embeds in a surface of Euler genus $g$ has at most $2(n+g-2)$ edges. 

Given an embedding of an $n$-vertex graph in a surface $\Sigma$ of Euler genus $g$ (where $n+g \geq 3$), we may add edges (if necessary) to obtain a pseudograph triangulation with exactly $3(n+g-2)$ edges. When we say that an edge-maximal embedding in $\Sigma$ or an edge-maximal graph embeddable in $\Sigma$ is ``$k$ edges short of a triangulation'' we mean that it has exactly $3(n+g-2)-k$ edges.

We need the following lemmas about edge-maximal embeddings. The first says that we may restrict our attention to 2-cell embeddings.

\begin{lem}
\label{2cell}
Let $c \geq 3$, and assume that for every surface $\Sigma$ of Euler genus $g$, every edge-maximal 2-cell embedding in $\Sigma$ is at most $cg$ edges short of a triangulation of $\Sigma$. Then for every surface $\Sigma$ of Euler genus $g$, every  edge-maximal embedding in $\Sigma$ is at most $cg$ edges short of a triangulation of $\Sigma$.
\end{lem}

\begin{proof}
Consider an edge-maximal embedding of a graph $G$ in some surface $\Sigma$ of Euler genus $g$. This embedding defines a combinatorial embedding of $G$, which corresponds to a 2-cell embedding in some surface $\Sigma'$ of Euler genus $g'\leq g$. If a non-edge of $G$ can be added to this embedding in $\Sigma'$, then the same non-edge can be added to the original embedding in $\Sigma$. Since the embedding in $\Sigma$ is edge-maximal, so too is the embedding in $\Sigma'$. By assumption, $G$ is at most $cg'$ edges short of a triangulation in $\Sigma'$. That is, $|E(G)| \geq 3(|V(G)|+g'-2) - cg' =  3(|V(G)|-2) - (c-3)g'\geq 3(|V(G)|-2) - (c-3)g = 3(|V(G)|+g-2) -cg$. That is, $G$ is at most $cg$ edges short of a triangulation in $\Sigma$.
\end{proof}

\begin{lem}
\label{MinDegree3}
Every graph $G$ with $n \geq 4$ vertices that has an edge-maximal 2-cell embedding in some surface is 3-connected.
\end{lem}

\begin{proof}
$G$ is connected since the embedding is edge-maximal and Euler genus is additive on components and blocks \citep{MoharThom}. If $G$ contains a vertex $v$ of degree $1$ and $vw$ is the edge incident to $v$, then $w$ has a distinct neighbour, so the facial walk starting with $vw$ is followed by $wx$ for some $x\not\in\{v,w\}$, and the edge $vx$ can be added to $G$, contradicting the edge-maximality of the embedding of $G$. Thus $G$ has minimum degree at least 2. Let $\pi_v$ denote the cyclic ordering of edges incident to each vertex $v$ in an embedding of $G$ in $\Sigma$. 


Suppose $G$ contains a vertex $v$ of degree $2$. Let $u$ and $w$ be the neighbours of $v$. We may assume that the edges $uv$ and $vw$ have signature +1. For clarity, observe that the edge $uw$ must be in $G$, with signature +1, since if not we could add it. Since $G$ is connected and $n \geq 4$, at least one of $u$ and $w$, say $w$, has a neighbour not in $\{u, v, w\}$.  Consider the cyclic order $\pi_w$: if $wu$ follows $wv$ then let $wx$ be the edge preceding $wv$, else let $wx$ be the edge following $wv$. Note that $x$ is not in $\{u, v, w\}$. We can add the edge $vx$, with signature +1, as follows.  Insert $vx$ in $\pi_v$ after $vw$ and insert $xv$ in $\pi_x$ before $xw$.  The original facial walk $W$ starting $x w v u \ldots$ is replaced by two facial walks $W_1 = x w v x$ and $W_2 = x v u \ldots$ where $W_2$ is obtained from $W$ by replacing the two-edge path $x w v$ by the single edge $x v$. By maximality, $G$ has minimum degree at least 3. 


Now we prove that for each vertex $v$ the subgraph induced on $N(v)$ has a Hamilton cycle ($G$ is ``locally Hamiltonian'').  Without loss of generality, the edges incident to $v$ have signature +1. Let $(vv_1, vv_2,\ldots,vv_d)$ be the cyclic ordering of the edges incident to $v$, where $d\geq 3$. We claim that $v_1v_2\ldots v_d$ is a cycle.  For suppose that say $v_1$ and $v_2$ are not adjacent. If $F$ is the face with facial walk starting $(v_1v, vv_2, \ldots)$, then we can add the edge $v_1v_2$ across $F$, which is a contradiction. Thus $G$ is locally Hamiltonian. 

Finally, any connected locally Hamiltonian graph is $3$-connected. Clearly $G$ cannot have a separating vertex.  Suppose $G$ has a separating pair of vertices $u, v$.  Thus $V(G) \setminus \{u,v\}$ can be partitioned into two non-empty parts $U$ and $W$ such that there are no $U$--$W$ edges.  Then $v$ must have a neighbour $a \in U$ and $b \in W$ (otherwise $u$ is a separating vertex) and there are two internally disjoint $ab$-paths in $G -v$ (around a Hamilton cycle in $N(v)$).  But both paths must go through $u$, a contradiction. Hence $G$ is 3-connected. 
\end{proof}

\begin{lem}
\label{FourDistinct}
Let $G$ be a graph with at least four vertices that has an edge-maximal 2-cell embedding in a surface. Then every non-triangular face contains four distinct vertices that are consecutive on the facial walk. Furthermore, for each string of six vertices that are consecutive on the facial walk, at least one of the three substrings of length 4 consists of distinct vertices.
\end{lem}

\begin{proof}
If $a,b,c$ are consecutive vertices on a face $F$, then $a,b,c$ are distinct, as otherwise $\deg(b)=1$, which would contradict \cref{MinDegree3}. Thus, if $F$ has length 4 or 5 then all the vertices on $F$ are distinct, and we are done. 
Now assume that $F$ has length at least 6. Let $v_1,\dots,v_6$ be consecutive vertices on $F$. 
If $v_1=v_4$ and $v_2=v_5$ and $v_3=v_6$, then the sequence is $v_1,v_2,v_3,v_1,v_2,v_3$, and the graph is $K_3$ (embedded in a non-orientable surface). Without loss of generality, $v_1\neq v_4$, implying $v_1,v_2,v_3,v_4$ are distinct.
\end{proof}

We noted earlier that Euler genus is additive on components and blocks. The main tool used in our proof is the following more general additivity theorem, proved independently  by several authors.

\begin{thm}[\cite{Miller-JCTB87,Archdeacon-JGT86,Richter-JCTB87}]
\label{Additivity}
If graphs $G_1$ and $G_2$ have at most two vertices in common, then the Euler genus of $G_1\cup G_2$ is at least the Euler genus of $G_1$ plus  the Euler genus of $G_2$.  
\end{thm}

Say a sequence  $v_1,\dots,v_s$ of vertices in a graph $G$ is \emph{ordered} if for each $i\in[1,s]$,
\[ \Bigg|N[v_i] \cap \Bigg( \bigcup_{j=1}^{i-1} N[v_j] \Bigg) \Bigg| \leq 2. \]
Here $N[v_i]$ is the closed neighbourhood $N(v_i)\cup\{v_i\}$. \cref{Additivity} implies the following result. 

\begin{cor}[\citep{JoretWood-JCTB10,NakaOta-JGT95}]
\label{AddCor}
If  $v_1,\dots,v_s$ is an ordered sequence of vertices in a graph $G$, and each $N[v_i]$ is a clique on at least five vertices, then the Euler genus of $G$ is at least $s$. 
\end{cor} 

We prove in \eqref{FindSequence} that given integers $g \geq 0$ and $s \geq 1$, there is an integer $b$ such that for every bipartite graph $G$ with Euler genus at most $g$, if $(A,B)$ is a  bipartition of $G$ such that $|B|>b$ and every vertex in $B$ has degree at most 4, then $B$ contains an ordered sequence of $s$ vertices. Let $f_g(s)$ be the least such integer $b$. 

We now give some illustrative examples. Since one vertex forms an ordered sequence, $f_g(1)=0$ for each $g \geq 0$. The planar bipartite graph $Q$ shown in \cref{Q} has a colour class $B$ with three vertices, each pair of which has three common neighbours. Thus $B$ contains no ordered sequence of length 2. Thus $f_0(2)\geq 3$. It is easily seen that $f_0(2)\leq 3$ (using a straightforward 
\begin{wrapfigure}{r}{49mm}
\vspace*{-1ex}
\centering
\includegraphics{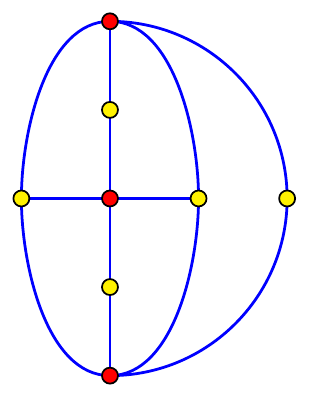}
\captionsetup{width=45mm}
\caption{\label{Q} The graph $Q$.}
\vspace*{-3ex}
\end{wrapfigure}
adaptation of the proof of  \cref{FindPair} below). Thus $f_0(2)=3$. Now consider general $g\geq 0$.  \citet{Ringel65a,Ringel65b} proved that the Euler genus of $K_{3,2g+2}$ equals $g$. If $B$ is the colour class of 
degree-$3$ vertices 
in $K_{3,2g+2}$, then every pair of vertices in $B$ have three common neighbours. Thus $B$ contains no ordered sequence of length $2$, and $f_g(2)\geq 2g+2$. \cref{FindPair} below proves this inequality is tight for $g\geq 1$. These constructions can be combined as follows. Fix $g\geq0$ and $s \geq 2$. Let $G$ be the graph obtained from $K_{3,2g+2}$ by adding $s-2$ disjoint copies of $Q$. Then $G$ has Euler genus $g$, and $G$ has a bipartition $(A,B)$ where $|B| = 2g+2+3(s-2)$ and every ordered sequence in $B$ has at most one vertex from each of the $s-1$ components of $G$. Thus $B$ contains no ordered sequence of length $s$, and 
\begin{equation}
\label{LowerBound}
f_g(s) \geq 2g+3s-4.
\end{equation}

The next lemma motivates the definition of $f_g(s)$.

\begin{lem}
\label{EdgesShort}
Every edge-maximal embedding of a graph $G$ in a surface $\Sigma$ of Euler genus $g  \geq1$ is at most $5f_g (g + 1) - 1$ edges short of a triangulation of $\Sigma$. 
\end{lem}

\begin{proof}
Note that $f_g(g+1)\geq 5g-1$ by \eqref{LowerBound}, which implies that $5f_g(g+1)-1 \geq 3g$. Thus, we may assume this embedding is 2-cell by \cref{2cell}. Let $n:=|V(G)|$. If $n \leq 7g$ then the number of edges in a triangulation, $3(n+g-2)$, is at most 
$24g -6< 5(5g-1)-1 \leq 5f_g(g+1)-1$ by \eqref{LowerBound}, and the result holds.  Now assume that $n\geq 7g+1\geq 8$.

By \cref{MinDegree3}, $G$ has minimum degree at least 3. We may assume the embedding of $G$ is not a triangulation.  
Let $G'$ be the embedded pseudograph obtained from $G$ as follows. 
Consider a face $F$ in $G$ with length $t\geq 4$. 
We shall add edges to $G$ across $F$ so that each of the resulting faces in $G'$ contains at least four distinct vertices. 
By \cref{FourDistinct}, $F$ contains four distinct consecutive vertices. 
Let  $(v_0,v_1,v_2,\dots,v_{t-1})$ be a facial walk of $F$, where $v_0,v_1,v_2,v_3$ are distinct. 
Add the edge $v_0v_i$ to $G$ whenever $i\equiv 3\pmod{5}$ and $3\leq i\leq t-5$; this divides $F$ into $\lfloor \frac{t+2}{5} \rfloor$ faces in $G'$ each containing at least four distinct vertices  (since $v_0,v_1,v_2,v_3$ are distinct, and every other face contains six consecutive vertices in $F$, and thus, by \cref{FourDistinct}, has at least four distinct vertices). 

For each non-triangular face $F$ of $G'$, add a vertex inside $F$ adjacent to four distinct vertices of $F$. Let $B$ be the set of these added vertices, and let $G''$ be the resulting embedded graph. Since the embedding of $G$ is edge-maximal, each face of $G$ induces a clique. Thus $N_{G''}[v]$ induces $K_5$ for each $v\in B$.

Consider a non-triangular face $F$ of length $t$ in $G$.
Then $B$ contains exactly $\floor{\frac{t+2}{5}}$ vertices corresponding to $F$. 
Note that $t-3 \leq 5\floor{\frac{t+2}{5}}-1$ edges are sufficient (and necessary) to triangulate $F$. 
Thus the embedding of $G$ can be extended to a triangulation by adding at most $5|B|-1$ edges. 

Let $G'''$ be the induced bipartite subgraph of $G''$ with bipartition $\{B,\cup_{v\in B}N_{G''}(v)\}$. 
By construction, $G'''$ embeds in $\Sigma$ and every vertex in $B$ has degree 4. 

Suppose for a contradiction that $|B|>f_g(g+1)$.
Thus $B$ contains an ordered sequence $v_1,\dots,v_{g+1}$ in $G''$. 
Since $N_{G''}[v_i]$ induces $K_5$, by \cref{AddCor}, the Euler genus of $G''$ is at least $g+1$, which is a contradiction. 
Thus  $|B|\leq f_g(g+1)$. Hence $G$ is at most $5f_g(g+1)-1$ edges short of a triangulation. 
\end{proof}

It remains to show how to find ordered sequences. The next lemma is useful.

\begin{lem}
\label{FindVertex}
Fix an integer $c\geq 7$. 
Let $G$ be a bipartite graph with bipartition $A,B$ and with Euler genus at most $g$. 
If $B$ is non-empty and $|B|> \tfrac{2c}{c-6}(g-2)$, then some vertex in $B$ has at most two neighbours with degree at least $c$.
\end{lem}

\begin{proof}
Let $A'$ be the set of vertices in $A$ with degree at least $c$. 
Suppose for a contradiction that every vertex in $B$ has at least three neighbours in $A'$. 
Double-counting the edges with endpoints in $A'$ and $B$ gives
$c|A'|\leq 2(|A'|+|B|+g-2)$ and $3|B|\leq 2(|A'|+|B|+g-2)$. 
Adding 2 times the first inequality plus $c-2$ times the second inequality gives
$|B|\leq \tfrac{2c}{c-6}(g-2)$, which is the desired contradiction. 
\end{proof}

We have the following recursive upper bound for $f_g(s)$. 

\begin{lem}
\label{Recurence}
For integers $g\geq 1$ and $s\geq 2$ and $c\geq 7$, 
\begin{equation*}
f_g(s) \leq \max\left\{\frac{2c}{c-6}(g-2),2c-3+f_g(s-1)\right\}.
\end{equation*}
\end{lem}

\begin{proof}
Let $G$ be a bipartite graph with Euler genus at most $g$, where  $(A,B)$ is  a bipartition of $G$ such that 
$|B|> \max\big\{\frac{2c}{c-6}(g-2),2c-3+f_g(s-1)\big\}$
and every vertex in $B$ has degree at most $4$. 
Our goal is to show  that $B$ contains an ordered sequence of $s$ vertices.
Since $B$ is non-empty and $|B|> \tfrac{2c}{c-6}(g-2)$, by \cref{FindVertex}, some vertex $v_s$ in $B$ has at most two neighbours with degree at least $c$. 
If $\deg(v_s)\leq 2$ then let $X:= \emptyset$. Otherwise, let $X$ be the set of neighbours of $v_s$ other than two of highest degree.  Thus $|X| \leq 2$ and each vertex $u \in X$ has degree at most $c-1$.
Let $G'$ be obtained from $G$ by deleting $N[u]$ for each $u\in X$. 
Let $A',B'$ be the bipartition of $G'$ inherited from $G$.
Note that 
$$|B'|\geq |B| -  (2c-3)  >  \max\left\{\tfrac{2c}{c-6}(g-2),2c-3+f_g(s-1)\right\} -(2c-3) \geq f_g(s-1).$$ 
Thus  $B'$ contains an ordered sequence $v_1,\dots,v_{s-1}$ in $G'$.  
By construction, $v_s$ has at most two neighbours in $G'$. 
Thus $v_1,\dots,v_s$ is an ordered sequence in $G$. 
\end{proof}

Since $f_g(1)=0$, \cref{Recurence} implies that for all integers $c\geq 7$ and  $s\geq 1$, 
\begin{equation}
\label{FindSequence}
f_g(s) \leq (2c-3)(s-2) + \max\big\{ \tfrac{2c}{c-6} (g-2), 2c-3\big\}.
\end{equation}
With any choice of $c\geq 7$, this implies that $f_g(g+1)$ is $O(g)$. \cref{EdgesShort} then implies that every edge-maximal embedding in a surface $\Sigma$ of Euler genus $g$ is  $O(g)$ edges short of a triangulation of $\Sigma$. Therefore the graphs embeddable in $\Sigma$ are $O(g)$-impure, which is the main result of this paper (\cref{Summary}). For example, with $c=8$ and $g\geq 4$, 
$$f_g(g+1) \leq 13(g-1) + \max\big\{ 8 (g-2), 13\big\} = 21g-29,$$ and by \cref{EdgesShort} every edge-maximal graph embeddable in a surface of Euler genus $g\geq 4$ is at most $105g-146$ edges short of a triangulation. 

\section{Improving the Constants}

To improve the constant in our main result, we first give a precise result for ordered sequences of length 2, improving on the bound in \eqref{FindSequence} with $s=2$.

%
%

\begin{lem}
\label{FindPair}
$f_g(2) = 2g+2$ for $g\geq 1$.
\end{lem}
 
\begin{proof}
We proved above that $K_{3,2g+2}$ shows that $f_g(2)\geq 2g+2$ for $g\geq 1$. We now prove the corresponding upper bound. 

Let $G$ be a  bipartite graph $G$ with Euler genus at most $g$. Assume that  $(A,B)$ is  a bipartition of $G$ such that every vertex in $B$ has degree at most $4$ and $|B|\geq 2g+3$. We claim that  $B$ contains an ordered sequence of two vertices. That is, $B$ contains two vertices with at most two common neighbours. Suppose for a contradiction that each pair of vertices in $B$ has at least three common neighbours.

By adding degree-1 vertices in $A$, we may assume that every vertex in $B$ has degree exactly $4$. Without loss of generality, $A = \bigcup_{b \in B} N(b)$. We have $4|B| \leq |E(G)| \leq 2(|A| + |B| + g-2)$ implying $|B| \leq |A| +g-2$ and 
$|A| \geq (2g+3)-(g-2)=g+5\geq 6$. 
 
Let $a, b \in B$ have $N(a) \neq N(b)$. Let $X = N(a) \cap N(b)$ and $Y = N(a) \cup N(b)$.  Then $|X|=3$ 
and $|Y|=5$. Let $N(a)=X \cup \{a'\}$ and $N(b)=X \cup \{b'\}$.  Since $|A| \geq 6$ there is a vertex $c \in B$ 
with $N(c)$ not contained in $Y$.  If $a' \in N(c)$ then $|N(c) \cap N(b)| \leq 2$, so $a' \not\in N(c)$; and similarly 
$b' \not\in N(c)$.  Hence $N(c) \cap Y = X$, so we may write $N(c)=X \cup \{c'\}$.  Note that $a',b',c'$ are distinct 
and not in $X$, so we have symmetry between $(a,a')$, $(b,b')$ and $(c,c')$.
 
Now consider any $v \in B$.  $N(v)$ cannot contain $\{a',b',c'\}$ (since then for example $|N(v) \cap N(a)| \leq 2$); so assume without loss of generality that $c' \not\in N(v)$.  But then we must have $N(v) \cap N(c)=X$, and so $N(v)$ contains $X$.  We have shown that $N(v)$ contains $X$ for each $v \in B$.  But now the induced bipartite graph with parts $X$ and $B$ is complete.  Hence
$3|B| \le 2(3+|B|+g-2)$, implying $|B| \le 2g+2 < 2g+3$.  This contradiction completes the proof.
\end{proof}

\cref{Recurence} and \cref{FindPair} imply that for $g\geq 1$ and $s\geq 2$, 
\begin{equation}
\label{fgb}
f_g(s) \leq
\begin{cases}
2g+2 & \text{ if }s=2,\\
\min\left\{\max\left\{\frac{2c_s}{c_s-6}(g-2),2c_s-3+f_g(s-1)\right\}:c_s\geq 7\right\} & \text{ if }s\geq 3.\\
\end{cases}
\end{equation}

For non-orientable surfaces, \cref{NonOrienTable} shows the optimal choice of $c_3,\dots,c_{g+1}$ in \eqref{fgb} for each value of $g\leq 20$, along with the corresponding lower bound on the number of edges in an edge-maximal graph.

\begin{table}[ht]
\centering
\caption{\label{NonOrienTable}Number of edges in an edge-maximal graph embeddable in a non-orientable surface.}
\begin{tabular}{ccccc}
\hline
$g$ & surface & $c_3,\dots,c_{g+1}$ &  impurity $\leq $ & $|E(G)|\geq$ \\
\hline
$1$ & $\mathbb{N}_{1}$ & $$ & $19$ & $3n-22$\\ 
$2$ & $\mathbb{N}_{2}$ & $7$ & $84$ & $3n-84$\\ 
$3$ & $\mathbb{N}_{3}$ & $7,7$ &  $149$ & $3n-146$\\ 
$4$ & $\mathbb{N}_{4}$ & $8,7,7$ &  $224$ & $3n-218$\\ 
$5$ & $\mathbb{N}_{5}$ & $8,8,7,7$ &  $299$ & $3n-290$\\ 
$6$ & $\mathbb{N}_{6}$ & $9,8,8,7,7$ &  $384$ & $3n-372$\\ 
$7$ & $\mathbb{N}_{7}$ & $9,8,8,7,7,7$ &  $459$ & $3n-444$\\ 
$8$ & $\mathbb{N}_{8}$ & $10,8,8,8,7,7,7$ & $534$ & $3n-516$\\ 
$9$ & $\mathbb{N}_{9}$ & $10,9,8,8,8,7,7,7$ &  $619$ & $3n-598$\\ 
$10$ & $\mathbb{N}_{10}$ & $10,9,8,8,8,8,7,7,7$ & $699$ & $3n-675$\\ 
$11$ & $\mathbb{N}_{11}$ & $11,9,8,8,8,8,8,7,7,7$ & $784$ & $3n-757$\\ 
$12$ & $\mathbb{N}_{12}$ & $11,9,9,8,8,8,8,7,7,7,7$ &  $864$ & $3n-834$\\ 
$13$ & $\mathbb{N}_{13}$ & $11,10,9,8,8,8,8,8,7,7,7,7$ &  $944$ & $3n-911$\\ 
$14$ & $\mathbb{N}_{14}$ & $12,10,9,8,8,8,8,8,8,7,7,7,7$ &  $1024$ & $3n-988$\\ 
$15$ & $\mathbb{N}_{15}$ & $12,10,9,9,8,8,8,8,8,8,7,7,7,7$ & $1109$ & $3n-1070$\\ 
$16$ & $\mathbb{N}_{16}$ & $12,10,9,9,8,8,8,8,8,8,8,7,7,7,7$ &  $1189$ & $3n-1147$\\ 
$17$ & $\mathbb{N}_{17}$ & $13,10,9,9,8,8,8,8,8,8,8,7,7,7,7,7$ & $1269$ & $3n-1224$\\ 
$18$ & $\mathbb{N}_{18}$ & $13,10,9,9,9,8,8,8,8,8,8,8,7,7,7,7,7$ &  $1359$ & $3n-1311$\\ 
$19$ & $\mathbb{N}_{19}$ & $13,11,10,9,9,8,8,8,8,8,8,8,8,7,7,7,7,7$ & $1439$ & $3n-1388$\\ 
$20$ & $\mathbb{N}_{20}$ & $13,11,10,9,9,8,8,8,8,8,8,8,8,8,7,7,7,7,7$ & $1519$ & $3n-1465$\\ 
\hline
\end{tabular}
\end{table}

The next lemma show a method for choosing the constants $c_s$ in \eqref{fgb}. All logarithms are natural.

\begin{lem}
\label{Technical}
Let $\lambda= 25 - 11\left( \tfrac{48332}{114345}+\tfrac{16}{33} \log 2\right) \approx 16.6533\ldots$ to four decimal places. 
Then for $g \geq 2$, 
$$f_g(g+1) \leq  \lambda (g-2) + 2 \Big\lceil\sqrt{\tfrac32(g-2)}\Big\rceil+33$$
\end{lem}

\begin{proof}
For $i\geq 7$, let 
$$\alpha_i := \sum_{j=i+1}^\infty\frac{12}{(j-7)(j-6)(2j-3)}.$$
Then $$0.758757\ldots =   \alpha_7 > \alpha_8 > \alpha_9 > \dots .$$

These numbers $\alpha_i$ are used below to calculate the values $c_s$ in \eqref{fgb}. 
For example, $\alpha_7\approx 0.76$ means that $c_s=7$ roughly for $0.76 g \leq s \leq g$, 
and $\alpha_8\approx 0.30$ means that $c_s=8$ roughly for $0.30 g \leq s \leq 0.76g$.
This behaviour is evident in the lower rows of \cref{NonOrienTable}. 
The definition of $\alpha_i$ is designed to minimise the ``max'' operation in \eqref{fgb}. 

We now upper bound $\alpha_k$. 
Since $(j-6)(2j-3)\geq 2(j-7)^2$ for $j\geq 7$,
\begin{align*}
\alpha_{k} 
 = \sum_{j=k+1}^\infty\frac{12}{(j-7)(j-6)(2j-3)}
 \leq \sum_{j=k+1}^\infty\frac{6}{(j-7)^3}
 \leq \int_{k+1}^\infty\frac{6}{(j-8)^3}dj
 = \frac{3}{(k-7)^2}
\enspace.
\end{align*}
With $k:= \CEIL{\sqrt{\frac32 (g-2)}}+7$ we have
$(k-7)^2\geq \frac32 (g-2)$ and $\alpha_{k}(g-2) \leq \frac{3}{(k-7)^2}(g-2)\leq2$. 
Let $k$ be the minimum integer such that $\alpha_{k}(g-2) \leq2$. 
Thus $k\leq \CEIL{\sqrt{\frac32 (g-2)}}+7$. 
For $i\in[7,k]$,  define 
$$\beta_i:=\ceil{\alpha_i(g-2)} \quad\text{and}\quad \gamma_i:=\beta_i-\alpha_i(g-2).$$

We claim that  $\beta_k=2$. If not, then $\alpha_k(g-2)\leq 1$ implying 
$$\frac{12}{(k-7)(k-6)(2k-3)} = \alpha_{k-1}(g-2) - \alpha_k(g-2) > 2-1 = 1,$$
which has no solution. Thus $\beta_k=2$.
Define $\beta_{2g+2}:=1$. 

For $i\in[7,2g+2]$, define 
$$L_i :=
\begin{cases}  
(\beta_{i}+1,\beta_{i}+2,\dots,\beta_{i-1}) & \text{ if }i\in[8,2g+2]\\
(\beta_{7}+1,\beta_{7}+2,\dots,g+1) & \text{ if } i=7.
\end{cases}$$
Then $L_{2g+2},\dots,L_7$ is a partition of  $[2,g+1]$. 
Define $\ell_i := |L_i|$. Then for $i\in[8,2g+2]$, 
\begin{align}
\label{elli}
\ell_i 
 = (\alpha_{i-1}-\alpha_i)(g-2)+(\gamma_{i-1} - \gamma_i) = \frac{12(g-2)}{(i-7)(i-6)(2i-3)}  +(\gamma_{i-1} - \gamma_i)
\end{align}
and
\begin{align}
\label{ell7}
\ell_7 = g+1 - \beta_7 
= g+1 - \alpha_7(g-2) - \gamma_7 
= (1 - \alpha_7)(g-2) - \gamma_7 + 3.
\end{align}
It may be that $\ell_i=0$ for some values of $i$. (For example, that there is no 12 in $c_3,\dots,c_{g+1}$ in the final row of \cref{NonOrienTable} corresponds to $\ell_{12}=0$.)\ If $\ell_i>0$ and $i<2g+2$, then let $i^*:=\min\{j>i:\ell_j>0\}$. 
Since $\ell_{2g+2}>0$ this is well-defined. Note that $\ell_j=0$ for $j\in[i+1,i^*-1]$ and $\beta_{i^*}+\ell_{i^*}=\beta_{i}$. 
For $s\in[2,g+1]$, there is a unique integer $i$ such that $\ell_i>0$ and $s\in L_i$, in which case define $c_s :=i$. 
Thus $c_s\geq 7$. 
Note that $s$ can be uniquely written $s=\beta_{i}+z$ for some $i\in[7,2g+2]$ with $\ell_i>0$ and $z\in[1,\ell_i]$. 
These definitions are summarised as follows.
\newpage
\begin{align*}
L_{2g+2} & = (2=\beta_{2g+2}+1)\\
L_{2g+1} & = \emptyset\\
& \vdots\\
L_{k+1} & = \emptyset\\
L_k & = (3=\beta_{k}+1,\beta_{k}+2,\dots,\beta_{k}+\ell_k=\beta_{k-1})\\
& \vdots\\
L_{i^*} & = (\beta_{i^*}+1,\beta_{i^*}+2,\dots,\beta_{i^*}+\ell_{i^*}=\beta_{i})\\
L_{i^*-1} & = \emptyset\\
& \vdots\\
L_{i+1} & = \emptyset\\
L_i & = (\beta_{i}+1,\beta_{i}+2,\dots,\beta_{i}+\ell_i=\beta_{i-1})\\
& \vdots\\
L_8 & = (\beta_{8}+1,\beta_{8}+2,\dots,\beta_8+\ell_8=\beta_{7})\\
L_7 & = (\beta_{7}+1,\beta_{7}+2,\dots,\beta_7+\ell_7=g+1).
\end{align*}
Define
\begin{equation*}
f'_g(s):=
\begin{cases}
2g+2 & \text{ if }s=2,\\
\max\left\{\frac{2c_s}{c_s-6}(g-2),2c_s-3+f'_g(s-1)\right\} & \text{ if }s\geq 3.\\
\end{cases}
\end{equation*}
It follows by induction on $s$ that $f_g(s)\leq f'_g(s)$. Thus to prove the desired upper bound on $f_g(s)$ it suffices to prove the same upper bound on $f'_g(s)$. It is helpful to note that $f'_g(s)$ is calculated by a row-by-row traversal of the above table, where the row corresponding to $L_i$ uses $c_s=i$ in the calculation of $f'_g(s)$. Thus for $s=\beta_i+z$ where $z\in[1,\ell_i]$, 
\begin{equation}
\label{Observe}
f'_g(\beta_i+z) = f'_g(\beta_i+1)+(z-1)(2i-3).
\end{equation}
Thus our focus is on estimating $f'_g(\beta_i+1)$, which equals $\max\big\{\frac{2i}{i-6}(g-2),2i-3+f'_g(\beta_i)\big\}$. 
In Claim~1 below we show that $\frac{2i}{i-6}(g-2)$ is `close' to $2i-3+f'_g(\beta_i)$. 
To do so, define the following recursive `error' function. 
First, let $E_{2g+2}:=0$ and let $E_k:=0$. Then for $i$ such that $\ell_i>0$, let
$$E_i := \max\left\{0 ,  \left( \sum_{j=i+1}^{i^*-1} 2\gamma_{j}\right)  + (2i-1)\gamma_i -(2i^*-3)\gamma_{i^*} + E_{i^*} \right\}.$$

\textbf{Claim 1.} For $s\in[2,g+1]$, if $s=\beta_i+z$ where $z\in[1,\ell_i]$,
$$f'_g(\beta_i+z) \leq \frac{2i}{i-6}(g-2)+(z-1)(2i-3)+E_i.$$

\begin{proof}
We proceed by induction on $s$. First consider the base case $s=2$. Then with $i=2g+2$ we have 
$s=\beta_{2g+2}+1= \frac{2i}{i-6}(g-2)$ and the claim holds with $E_{2g+2}=0$. 

Now assume that $s\geq 3$ and the claim holds for $s-1$. 
By \eqref{Observe}, it suffices to consider the $z=1$ case, and we may assume that $\ell_i>0$. 
Then $s-1=\beta_i=\beta_{i^*}+\ell_{i^*}$. 
By induction, 
\begin{align*}
 f'_g(\beta_{i^*}+\ell_{i^*}) 
& \leq \frac{2i^*}{i^*-6}(g-2)+(\ell_{i^*}-1)(2i^*-3)+E_{i^*}\\
& = \frac{2i^*}{i^*-6}(g-2)+\ell_{i^*}(2i^*-3) - (2i^*-3) +E_{i^*}
\end{align*}
Since $\ell_j=0$ for $j\in[i+1,i^*-1]$,
\begin{align*}
 f'_g(\beta_{i^*}+\ell_{i^*})  \leq   \frac{2i^*}{i^*-6}(g-2)+\left( \sum_{j=i+1}^{i^*}(2j-3)\ell_j \right) - (2i^*-3) + E_{i^*}.
\end{align*}
By \eqref{elli} and since $2i-3\leq 2i^*-3$, 
\begin{align*}
& \quad f'_g(\beta_{i^*}+\ell_{i^*})  + 2i-3 \\
 & \leq \frac{2i^*}{i^*-6}(g-2)+ \sum_{j=i+1}^{i^*}(2j-3)\left( \frac{12(g-2)}{(j-7)(j-6)(2j-3)} + \gamma_{j-1}-\gamma_j \right) + E_{i^*}\\
& = \frac{2i^*}{i^*-6}(g-2)+ \left( \sum_{j=i+1}^{i^*} \frac{12(g-2)}{(j-7)(j-6)} \right) + \left( \sum_{j=i+1}^{i^*} (\gamma_{j-1}-\gamma_j)(2j-3)  \right)  + E_{i^*}\\
& = \frac{2i^*}{i^*-6}(g-2)+ (g-2)\left( \sum_{j=i+1}^{i^*} \frac{2(j-1)}{(j-1)-6} - \frac{2j}{j-6} \right) + \left( \sum_{j=i+1}^{i^*} (\gamma_{j-1}-\gamma_j)(2j-3)\right)  + E_{i^*}\\
& = \frac{2i^*}{i^*-6}(g-2)+ (g-2)\left(  \frac{2i}{i-6} - \frac{2i^*}{i^*-6} \right) + \left( \sum_{j=i+1}^{i^*-1} 2\gamma_{j}\right)  + (2i-1)\gamma_i -\gamma_{i^*}(2i^*-3) + E_{i^*}\\
& =    \frac{2i}{i-6}(g-2)  + \left( \sum_{j=i+1}^{i^*-1} 2\gamma_{j}\right)  + (2i-1)\gamma_i -(2i^*-3)\gamma_{i^*}  + E_{i^*}.
\end{align*}
Since $c_s=i$ and by \eqref{elli}, 
\begin{align*}
 f'_g(s) 
& = \max\left\{\frac{2i}{i-6}(g-2),2i-3+f'_g(\beta_{i^*}+\ell_{i^*})\right\} \\
& \leq \max\left\{\frac{2i}{i-6}(g-2),\,\frac{2i}{i-6}(g-2)  + \left( \sum_{j=i+1}^{i^*-1} 2\gamma_{j}\right)  + (2i-1)\gamma_i -(2i^*-3)\gamma_{i^*} + E_{i^*} \right\} \\
& = \frac{2i}{i-6}(g-2) + \max\left\{0, \left( \sum_{j=i+1}^{i^*-1} 2\gamma_{j}\right)  + (2i-1)\gamma_i -(2i^*-3)\gamma_{i^*} + E_{i^*} \right\} \\
& = \frac{2i}{i-6}(g-2) + E_{i} .
\end{align*}
This completes the proof of the claim. 
\end{proof}

We now upper bound the $E_i$.

\textbf{Claim 2.} For $i\in[7,k]$ such that $\ell_i>0$, there are integers $\delta_i,\dots,\delta_k$, such that 
$$E_i\leq \sum_{j=i}^k \delta_j \gamma_j,$$
and if $\Delta_i$ is the multiset $\{\delta_j\geq 0: j\in[i,k]\}$, then $\sum\Delta_i\leq 2k-3$. Moreover, if $E_i>0$ then $\delta_i=2i-1$.

\begin{proof}
We proceed by induction on $i=k,k-1,\dots,2$. In the base case $i=k$, we have $E_k=0$ and the claim holds with $\delta_k=0$ and $X_k=0$. Now assume that $i\in[7,k-1]$ with $\ell_i>0$ and the claim holds for $i^*$. Thus, there are integers 
$\delta_{i^*},\dots,\delta_k$, such that 
$$E_{i^*}\leq \sum_{j=i^*}^k \delta_j \gamma_j,$$
and $\sum\Delta_{i^*}\leq 2k-3$. Moreover, if $E_{i^*}>0$ then $\delta_{i^*}=2i^*-1$. 
By definition, 
$$E_i = \max\left\{0 , (2i-1)\gamma_i + \left( \sum_{j=i+1}^{i^*-1} 2\gamma_{j}\right)   -(2i^*-3)\gamma_{i^*} + E_{i^*} \right\}.$$
If $E_i=0$ then the claim holds with $\delta_i,\dots,\delta_k=0$. 
Now assume that $E_i>0$. 

First suppose that $E_{i^*}=0$. Then
$$E_i = (2i-1)\gamma_i + \left( \sum_{j=i+1}^{i^*-1} 2\gamma_{j}\right)   -(2i^*-3)\gamma_{i^*} .$$
and the claim holds with $\delta_i=2i-1$ and $\delta_{i^*}=-(2i^*-3)$ and $\delta_j=2$ for $j\in[i+1,i^*-1]$, 
in which case $\Delta_i=\{2i-1,(i^*-1-i)\times 2\}$ and $\sum\Delta_i=2i^*-3\leq 2k-3$. 

Now assume that $E_{i^*}>0$. Then $\delta_{i^*}=2i^*-1$ and 
\begin{align*}
E_i 
& \leq  (2i-1)\gamma_i + \left( \sum_{j=i+1}^{i^*-1} 2\gamma_{j}\right)   -(2i^*-3)\gamma_{i^*} + \left(\sum_{j=i^*}^k \delta_j \gamma_j\right)\\
& = (2i-1)\gamma_i + \left( \sum_{j=i+1}^{i^*-1} 2\gamma_{j}\right)   + ((2i^*-1)-(2i^*-3))\gamma_{i^*} + \left(\sum_{j=i^*+1}^k \delta_j \gamma_j\right)\\
& = (2i-1)\gamma_i + \left( \sum_{j=i+1}^{i^*-1} 2\gamma_{j}\right)  + 2\gamma_{i^*} + \left(\sum_{j=i^*+1}^k \delta_j \gamma_j\right).
\end{align*}
Let $\delta_i:=2i-1$ and $\delta_{i^*}:=2$ and $\delta_j:=2$ for $j\in[i+1,i^*-1]$, 
Observe that $$\Delta_i =( \Delta_{i+1} \setminus \{ 2i^*-1\} ) \cup \{2i-1,2,(i^*-1-i)\times 2\}.$$
Thus $\sum\Delta_{i+1} = \sum\Delta_{i}$, which is at most $2k-3$ by assumption. 
Thus the claim is satisfied. 
\end{proof}

Claim 2 with $i=7$ implies that there are integers $\delta_7,\dots,\delta_k$, such that 
$$E_7\leq \sum_{j=7}^k \delta_j \gamma_j,$$
and $\sum\Delta_i\leq 2k-3$. Since $\gamma_j\in[0,1)$, 
$$E_7\leq \sum\Delta_7  \leq 2k-3.$$
Claim 1 and Equation~\eqref{ell7} then imply that for $s=g+1=\beta_7+\ell_7$,
\begin{align*}
f_g(g+1) \leq f'_g(g+1) & \leq 14(g-2)+11( \ell_7-1)+E_7\\
& \leq 14(g-2)+11( (1 - \alpha_7)(g-2) - \gamma_7 + 3  -1)+(2k-3)\\
& \leq 14(g-2)+11( (1 - \alpha_7)(g-2) +2 )+(2k-3)\\
& = (25 - 11\alpha_7)(g-2) +2k+19\\
& = (25 - 11( \tfrac{48332}{114345}+\tfrac{16}{33} \log 2))(g-2) +2k+19\\
& = \lambda (g-2) +2 \Big\lceil\sqrt{\tfrac{3}{2}(g-2)}\Big\rceil +33.
\end{align*}
This completes the proof. 
\end{proof}

Note that \eqref{LowerBound} implies that $f_g(g+1) \geq 5g-1$. Since $ \lambda< \tfrac{50}3$, this shows that \cref{Technical}  is within a factor of $\frac{10}{3}$ of optimal.

\begin{thm}
\label{NonOrient}
For every surface $\Sigma$ of Euler genus $g$, every edge-maximal embedding of a graph in $\Sigma$ is at most $84g$ edges short of a triangulation of $\Sigma$.
\end{thm}

\begin{proof}
By \cref{EdgesShort}, it suffices to show that $5f_g(g+1)-1 \leq 84g$. For $g\leq 299$, this is verified by direct calculation of the upper bound on $f_g(g+1)$ in \eqref{fgb}. For $g\geq 300$, by \cref{Technical}, 
\begin{equation*}
5 f_g(g+1)-1 
\leq 5 \Big(16.6534 (g-2) + 2 \big( 1  + \sqrt{\tfrac32(g-2)}\big)+33\Big) -1 
 \leq 84g.\qedhere
 \end{equation*}
\end{proof}

Note that for each surface $\Sigma$ of Euler genus $g$, \cref{planar} provides examples of edge-maximal 2-cell embeddings of graphs in $\Sigma$ that are $3g$ edges short of a triangulation of $\Sigma$. Thus the 84 in \cref{NonOrient} cannot be reduced to less than 3. Also note that $K_3$, which is edge-maximal embeddable on any surface $\Sigma$, is $3g$ edges short of a triangulation of $\Sigma$ (since every 3-vertex pseudograph triangulation of $\Sigma$ has $3g + 3$ edges).

\subsection{Orientable Surfaces}

Further improvements are possible if we restrict our attention to orientable surfaces. Let $G$ be an edge-maximal graph embeddable in an orientable surface $\Sigma$. Recall from \cref{FourDistinct}  that among six consecutive vertices on a face of $G$, there are at least four distinct vertices, as otherwise a facial walk would contain $abcabc$, implying $\deg(b)=2$. When  $\Sigma$ is orientable, among five consecutive vertices on a face of $G$, there are at least four distinct vertices, as otherwise a facial walk would contain $abcab$, repeating $ab$. This enables us to add more edges to $G'$ in the proof of \cref{EdgesShort}. Consider a face $F$ of $G$ of length $t\geq 4$.  By \cref{FourDistinct}, $F$ contains four distinct consecutive vertices. Let  $(v_0,v_1,v_2,\dots,v_{t-1})$ be a facial walk of $F$, where $v_0,v_1,v_2,v_3$ are distinct. Add the edge $v_0v_i$ to $G'$ whenever $i\equiv 3\pmod{4}$ and $3\leq i\leq t-4$; this divides $F$ into $\lfloor \frac{t+1}{4} \rfloor$ faces in $G'$ each containing four distinct vertices (since $v_0,v_1,v_2,v_3$ are distinct, and every other face contains five consecutive vertices in $F$, and thus has at least four distinct vertices). Define the graph $G''$ and set $B$  as above. Consider a face $F$ of $G$ of length $t\geq 4$. Then $B$ contains exactly $\floor{\frac{t+1}{4}}$ vertices corresponding to $F$, and $t-3 \leq 4\floor{\frac{t+1}{4}}-1$. Thus $G$ can be triangulated by adding at most $4|B|-1$ edges. By the same argument used in the proof of \cref{EdgesShort},  $G$ is at most $4f_g(g+1)-1$ edges short of a triangulation. This leads to the results shown in \cref{OrienTable} (by \cref{fgb}) and the following theorem.

\begin{thm}
\label{orientable}
For every orientable surface $\Sigma$ of Euler genus $g$, every edge-maximal embedding of a graph in $\Sigma$ is at most $67g$ edges short of a triangulation of $\Sigma$.
\end{thm}

\begin{proof}
By the above discussion it suffices to show that $4f_g(g+1)-1 \leq 67g$. For $g\leq 670$, this is verified by direct calculation of the upper bound on $f_g(g+1)$ in \eqref{fgb}. For $g\geq 671$, by \cref{Technical}, 
\begin{equation*}
4 f_g(g+1)-1 
\leq 4 \Big(16.6534 (g-2) + 2 \big( 1  + \sqrt{\tfrac32(g-2)}\big)+33\Big) -1 
 \leq 67g.\qedhere
 \end{equation*}
\end{proof}

\begin{table}[ht]
\centering
\caption{\label{OrienTable} Number of edges in an edge-maximal graph embeddable in an orientable surface.}
\begin{tabular}{ccccc}
\hline
$g$ & surface & impurity $\leq$ & $|E(G)|\geq$ \\
\hline
$2$ & $\mathbb{S}_1$ &  $67$ & $3n-67$\\ 
$4$ & $\mathbb{S}_2$ &  $179$ & $3n-173$\\ 
$6$ & $\mathbb{S}_3$ & $307$ & $3n-295$\\ 
$8$ & $\mathbb{S}_4$ & $427$ & $3n-409$\\ 
$10$ & $\mathbb{S}_5$ &  $559$ & $3n-535$\\ 
$12$ & $\mathbb{S}_6$ &  $691$ & $3n-661$\\ 
$14$ & $\mathbb{S}_7$ & $819$ & $3n-783$\\ 
$16$ & $\mathbb{S}_8$ &  $951$ & $3n-909$\\ 
$18$ & $\mathbb{S}_9$ &  $1087$ & $3n-1039$\\ 
$20$ & $\mathbb{S}_{10}$ & $1215$ & $3n-1161$\\ 
$22$ & $\mathbb{S}_{11}$  & $1339$ & $3n-1279$\\ 
$24$ & $\mathbb{S}_{12}$ & $1483$ & $3n-1417$\\ 
$26$ & $\mathbb{S}_{13}$ & $1607$ & $3n-1535$\\ 
$28$ & $\mathbb{S}_{14}$ & $1743$ & $3n-1665$\\ 
$30$ & $\mathbb{S}_{15}$ & $1875$ & $3n-1791$\\ 
$32$ & $\mathbb{S}_{16}$ & $2007$ & $3n-1917$\\ 
$34$ & $\mathbb{S}_{17}$ & $2139$ & $3n-2043$\\ 
$36$ & $\mathbb{S}_{18}$ & $2275$ & $3n-2173$\\ 
$38$ & $\mathbb{S}_{19}$ & $2411$ & $3n-2303$\\ 
$40$ & $\mathbb{S}_{20}$ & $2539$ & $3n-2425$\\ 
\hline
\end{tabular}
\end{table}

\section{Open Problems}

We conclude the paper with a few open problems. 
\begin{itemize}

\item
Let $c_1$ be the infimum of all numbers $c$ such that every edge-maximal graph embeddable in a
surface $\Sigma$ of Euler genus $g$ is at most $cg$ edges short of a triangulation of $\Sigma$. 
Let $c_2$ be the infimum of all numbers $c$ such that every edge-maximal embedding in a
surface $\Sigma$ of Euler genus $g$ is at most $cg$ edges short of a triangulation of $\Sigma$.
Trivially, $c_1 \leq c_2$.  We have proved that $3 \leq c_1 \leq c_2 <  84$. Can these inequalities be improved?

\item Are projective planar graphs pure?  Are there examples, other than $K_8-E(C_5)$, showing that the class of graphs embeddable in a given surface is impure?  
\item For a surface $\Sigma$, what is the least number $k$ such that for every edge-maximal graph $G$ embeddable in $\Sigma$, there is a triangulation $G'$ of $\Sigma$ with the same vertex set as $G$ such that $E(G)$ and $E(G')$ have symmetric difference of size at most $k$? 
\item If $G$ is embeddable in a surface $\Sigma$, and has sufficiently many vertices but is not edge-maximal, can one always add edges to obtain a triangulation of $\Sigma$?
\end{itemize}

\subsection*{Acknowledgements}

This research was initiated at the 2016 Barbados Graph Theory Workshop and the 2016 Workshop on Probability, Combinatorics and Geometry, both held at Bellairs Research Institute in Barbados. Thanks to the workshop organisers, and to the other participants for creating a stimulating working environment. Thanks to Vida Dujmovi\'c for helpful conversations about this research. 


\def\soft#1{\leavevmode\setbox0=\hbox{h}\dimen7=\ht0\advance \dimen7
  by-1ex\relax\if t#1\relax\rlap{\raise.6\dimen7
  \hbox{\kern.3ex\char'47}}#1\relax\else\if T#1\relax
  \rlap{\raise.5\dimen7\hbox{\kern1.3ex\char'47}}#1\relax \else\if
  d#1\relax\rlap{\raise.5\dimen7\hbox{\kern.9ex \char'47}}#1\relax\else\if
  D#1\relax\rlap{\raise.5\dimen7 \hbox{\kern1.4ex\char'47}}#1\relax\else\if
  l#1\relax \rlap{\raise.5\dimen7\hbox{\kern.4ex\char'47}}#1\relax \else\if
  L#1\relax\rlap{\raise.5\dimen7\hbox{\kern.7ex
  \char'47}}#1\relax\else\message{accent \string\soft \space #1 not
  defined!}#1\relax\fi\fi\fi\fi\fi\fi}

\end{document}